\documentclass[12pt]{article}

\usepackage{latexsym,amsmath,amsfonts,amsthm,amssymb}

\usepackage{url}

\newtheorem{lemma}{Lemma}[section]

\newcommand{\T}{\mathbb{T}}
\newcommand{\R}{\mathbb{R}}
\newcommand{\Z}{\mathbb{Z}}

\newcommand{\rar}{\mbox{$\rightarrow$}}

\begin{document}

\title{Leitmann's direct method of optimization\\
for absolute extrema of certain problems\\
of the calculus of variations on time scales\thanks{Dedicated to
George Leitmann on the occasion of his 85th birthday. The authors
are grateful to George Leitmann for many helpful remarks and
discussions. Accepted for publication (9/January/2010) 
in \emph{Applied Mathematics and Computation}.}}

\author{Agnieszka B. Malinowska\thanks{On leave
        of absence from Bia{\l}ystok University of Technology,
        Poland (abmalina@wp.pl).}\\
\texttt{abmalinowska@ua.pt}
\and Delfim F. M. Torres\\
\texttt{delfim@ua.pt}}

\date{Department of Mathematics\\
University of Aveiro\\
3810-193 Aveiro, Portugal}

\maketitle

%-----------------------------------------

\begin{abstract}
The fundamental problem of the calculus of variations on time scales
concerns the minimization of a delta-integral over all trajectories
satisfying given boundary conditions. This includes the
discrete-time, the quantum, and the continuous/classical
calculus of variations as particular cases.
In this note we follow Leitmann's direct method to
give explicit solutions for some concrete optimal 
control problems on an arbitrary time scale.
\end{abstract}

\smallskip


\smallskip

%-----------------------------------------

\smallskip


%-----------------------------------------

\section{Introduction}

The calculus on time scales is a recent field that unifies the
theories of difference and differential equations. It has found
applications in several contexts that require simultaneous modeling
of discrete and continuous data, and is nowadays under strong
current research in several different areas \cite{book:ts,book:ts1}.

The area of the calculus of variations on time scales, 
which we are concerned in this paper, 
was born in 2004 \cite{B:04} and is now receiving a lot of attention,
both for theoretical and practical reasons -- see
\cite{A:T,A:B:L:06,A:U:08,B:T:08,F:T:07,F:T:08,AM:T,Basia:post_doc_Aveiro:2,NM:T}
and references therein. Although the theory is already well
developed in many directions, a crucial problem still persists:
solving Euler-Lagrange delta-differential equations on arbitrary
time scales is difficult or even impossible. As a consequence, there
is a lack of concrete variational problems for which a solution is
known. In this paper we follow a different approach. We show that
the direct method introduced by Leitmann in the sixties of the XX
century \cite{Leitmann67} can also be applied to variational
problems on time scales. Leitmann's method is a venerable forty
years old method that has shown through the times to be an universal
and useful method in several different contexts -- see,
\textrm{e.g.},
\cite{Car,CarlsonLeitmann05a,CarlsonLeitmann05b,TE,Leit,Leitmann01,MR1954118,MR2065731,MR2035262,SilvaTorres06,withLeitmann,Wagener}.
Here we provide concrete examples of problems of the calculus of
variations on time scales for which a global minimizer is easily
found by the application of Leitmann's direct approach.

% --------------------------------------------------

\section{Preliminaries on time scales}

A {\it time scale} $\T$ is an arbitrary nonempty closed subset of
the set $\R$ of real numbers. It is a model of time. Besides
standard cases of $\mathbb{R}$ (continuous time) and $\mathbb{Z}$
(discrete time), many different models are used. For each time scale
$\mathbb{T}$ the following operators are used:
\begin{itemize}

\item the {\it forward jump operator} $\sigma:\T \rar \T$,
$\sigma(t):=\inf\{s \in \T:s>t\}$ for $t<\sup \T$ and
$\sigma(\sup\T)=\sup\T$ if $\sup\T<+\infty$;

\item the {\it backward jump operator} $\rho:\T \rar \T$,
$\rho(t):=\sup\{s \in \T:s<t\}$ for $t>\inf \T$ and
$\rho(\inf\T)=\inf\T$ if $\inf\T>-\infty$;

\item the {\it graininess function} $\mu:\T \rar [0,\infty)$,
$\mu(t):=\sigma(t)-t$.
\end{itemize}

For $\T=\R$ one has $\sigma(t)=t=\rho(t)$ and $\mu(t) \equiv 0$ for
any $t \in \R$. For $\T=\Z$ one has $\sigma(t)=t+1$, $\rho(t)=t-1$,
and $\mu(t) \equiv 1$ for every $t \in \Z$.

A point $t \in \mathbb{T}$ is called: (i) {\it{ right-scattered}} if
$\sigma(t)>t$, (ii) {\it{right-dense}} if $\sigma(t)=t$, (iii) {\it
{left-scattered}} if $\rho(t)<t$, (iv) {\it{left-dense}} if
$\rho(t)=t$, (v) {\it {isolated}} if it is both left-scattered and
right-scattered, (vi) {\it{dense}} if it is both left-dense and
right-dense. If $\sup \T$ is finite and left-scattered we set
$\mathbb{T}^\kappa :=\mathbb{T}\setminus \{\sup\T\}$; otherwise,
$\mathbb{T}^\kappa :=\mathbb{T}$.

We assume that a time scale $\mathbb{T}$ has the topology that it
inherits from the real numbers with the standard topology. Let $f:\T
\rar \R$ and $t \in \T^\kappa$. The {\it delta derivative} of $f$ at
$t$ is the real number $f^{\Delta}(t)$ with the property that given
any $\varepsilon$ there is a neighborhood $U=(t-\delta,t+\delta)
\cap \T$ of $t$ such that
\[|(f(\sigma(t))-f(s))-f^{\Delta}(t)(\sigma(t)-s)| \leq \varepsilon|\sigma(t)-s|\]
for all $s \in U$. We say that $f$ is {\it delta-differentiable} on
$\T$ provided $f^{\Delta}(t)$ exists for all $t \in \T^\kappa$.

We note that if $\T=\R$, then $f:\R \rar \R$ is delta differentiable
at $t \in \R$ if and only if $f$ is differentiable in the ordinary
sense at $t$. Then, $f^{\Delta}(t)=f'(t)$. If $\T=\Z$, then $f:\Z
\rar \R$ is always delta differentiable at every $t \in \Z$ with
$f^{\Delta}(t)=f(t+1)-f(t)$.

A function $f:\mathbb{T} \to \mathbb{R}$ is called
{\emph{rd-continuous}} if it is continuous at the right-dense points
in $\mathbb{T}$ and its left-sided limits exist at all left-dense
points in $\mathbb{T}$.
The set of all rd-continuous functions is denoted by $C_{rd}$.
Similarly, $C^1_{rd}$ will denote the set of functions from $C_{rd}$
whose delta derivative belongs to $C_{rd}$. A continuous function
$f$ is \emph{piecewise rd-continuously delta-differentiable} (we
write $f\in C_{prd}^{1}$) if $f$ is continuous and $f^{\Delta}$
exists for all, except possibly at finitely many $t \in
\T^{\kappa}$, and $f^{\Delta}\in C_{rd}$. It is known that 
piecewise rd-continuous functions possess an \emph{antiderivative},
\textrm{i.e.}, there exists a function $F$ with $F^{\Delta}=f$, and
in this case the \emph{delta-integral} is defined by
$\int_{c}^{d}f(t)\Delta t=F(d)-F(c)$ for all $c,d\in\T$. If
$\mathbb{T}=\mathbb{R}$, then \[\int\limits_{a}^{b}  f(t) \Delta
t=\int\limits_{a}^{b}  f(t) d t,\] where the integral on the right
hand side is the usual Riemann integral; if
$\mathbb{T}=h\mathbb{Z}$, $h>0$, and $a<b$,
then \[\int\limits_{a}^{b}  f(t) \Delta
t=\sum\limits_{k=\frac{a}{h}}^{\frac{b}{h}-1}  h\cdot f(kh)\, .\]
For $f:\T \rar X$, where $X$ is an arbitrary set, we define
$f^\sigma:=f\circ\sigma$.

% --------------------------------------------------

\section{Leitmann's direct method on time scales}

Let $\T$ be a time scale with at least two points. 
Throughout we let $a,b\in \T$ with $a<b$. 
For an interval $[a,b]\cap \T$ we simply write $[a,b]$.

The problem of the calculus of variations on time scales consists
in minimizing
\begin{equation*}
\mathcal{L}[x]=\int_{a}^{b}L(t,x^{\sigma}(t),x^{\Delta}(t))
\Delta t
\end{equation*}
over all $x\in C_{prd}^{1}([a,b],\R)$
satisfying the boundary conditions
\begin{equation}
\label{bc} x(a)=\alpha,\,  \ x(b)=\beta,
\end{equation}
where $\alpha$, $\beta\in \R$ and
$L:[a,b]^{\kappa}\times \R \times \R \rightarrow \R$. We
assume that $(t,y,v) \rightarrow L(t,y,v)$ has partial continuous
derivatives $L_{y}$ and $L_{v}$, respectively with respect to the
second and third arguments, for all $t\in[a,b]^{\kappa}$, and
$L(\cdot,y,v)$, $L_{y}(\cdot,y,v)$ and $L_{v}(\cdot,y,v)$ are
piecewise rd-continuous in $t$ for all $x\in C_{prd}^{1}([a,b],\R)$. A
function $x\in C_{prd}^{1}([a,b],\R)$
is said to be admissible if it
satisfies the boundary conditions \eqref{bc}.

Let $\tilde{L}:[a,b]^{\kappa}\times \R \times \R \rightarrow \R$. We
assume that $(t,y,v) \rightarrow \tilde{L}(t,y,v)$ has partial
continuous derivatives $\tilde{L}_{y}$ and $\tilde{L}_{v}$,
respectively with respect to the second and third arguments, for all
$t\in[a,b]^{\kappa}$, and $\tilde{L}(\cdot,y,v)$,
$\tilde{L}_{y}(\cdot,y,v)$ and $\tilde{L}_{v}(\cdot,y,v)$ are
piecewise rd-continuous in $t$ for all $x\in C_{prd}^{1}([a,b],\R)$. Consider
the integral
\begin{equation*}
\tilde{\mathcal{L}}[\tilde{x}]=\int_{a}^{b}\tilde{L}(t,\tilde{x}^{\sigma}(t),\tilde{x}^{\Delta}(t))
\Delta t
\end{equation*}

\begin{lemma}[Leitmann's fundamental lemma]
\label{Fund:lemma:Leit} Let $x=z(t,\tilde{x})$ be a transformation
having an unique inverse $\tilde{x}=\tilde{z}(t,x)$ for all $t\in
[a,b]$ such that there is a one-to-one correspondence
\begin{equation*}
    x(t)\Leftrightarrow \tilde{x}(t),
\end{equation*}
for all functions $x\in C_{prd}^{1}([a,b],\R)$ satisfying \eqref{bc} and all
functions $\tilde{x}\in C_{prd}^{1}([a,b],\R)$ satisfying
\begin{equation}\label{bc:trans}
\tilde{x}=\tilde{z}(a,\alpha), \quad \tilde{x}=\tilde{z}(b,\beta).
\end{equation}
If the transformation $x=z(t,\tilde{x})$ is such that there exists a
function $G:[a,b] \times \R \rightarrow \R$ satisfying the
functional identity
\begin{equation}\label{id}
L(t,x^{\sigma}(t),x^{\Delta}(t))-\tilde{L}(t,\tilde{x}^{\sigma}(t),\tilde{x}^{\Delta}(t))
=G^{\Delta}(t,\tilde{x}(t)) \, ,
\end{equation}
then if  $\tilde{x}^{*}$ yields the extremum of
$\tilde{\mathcal{L}}$ with $\tilde{x}^{*}$ satisfying
\eqref{bc:trans}, $x^{*}=z(t,\tilde{x}^{*})$ yields the extremum of
$\mathcal{L}$ for $x^{*}$ satisfying \eqref{bc}.
\end{lemma}

\begin{proof}
The proof is similar in spirit to Leitmann's proof
\cite{Leitmann67,Leit,Leitmann01, MR2035262}. Let  $x\in
C_{prd}^{1}([a,b],\R)$ satisfy \eqref{bc} and define functions $\tilde{x}\in
C_{prd}^{1}([a,b],\R)$ through the formula $\tilde{x}=\tilde{z}(t,x)$, $a\leq
t\leq b$. Then $\tilde{x}\in C_{prd}^{1}([a,b],\R)$ and satisfies
\eqref{bc:trans}. Moreover, as a result of \eqref{id}, it follows
that
\begin{equation*}
\begin{split}
\mathcal{L}[x]-\tilde{\mathcal{L}}[\tilde{x}] &=\int_{a}^{b}L(t,x^{\sigma}(t),x^{\Delta}(t))
\Delta
t-\int_{a}^{b}\tilde{L}(t,\tilde{x}^{\sigma}(t),\tilde{x}^{\Delta}(t))
\Delta t\\
&=\int_{a}^{b}G^{\Delta}(t,\tilde{x}(t)) \Delta t
=G(b,\tilde{x}(b))-G(a,\tilde{x}(a))\\
&=G(b,\tilde{z}(b,\beta))-G(a,\tilde{z}(a,\beta)),
\end{split}
\end{equation*}
from which the desired conclusion follows immediately since the
right-hand side of the above equality is a constant depending 
only on the fixed-endpoint conditions \eqref{bc}.
\end{proof}

% ------------------

\section{An illustrative example}

Let $a, b \in \T$, $a < b$, and $\alpha$ and $\beta$ be two given
reals, $\alpha \ne \beta$. We consider the following problem of the
calculus of variations on time scales:
\begin{equation}
\label{illust:Ex:mod}
\begin{gathered}
\text{minimize} \quad \mathcal{L}[x] =\int_{a}^{b}
\left((x^{\Delta}(t))^2
+x^{\sigma}(t) +t x^{\Delta}(t)\right) \Delta t \, , \\
x(a)=\alpha \, , \quad x(b)=\beta \, .
\end{gathered}
\end{equation}
We transform problem \eqref{illust:Ex:mod} into the trivial problem
\begin{gather*}
 \text{minimize} \quad
 \tilde{\mathcal{L}}[\tilde{x}]=\int_{a}^{b}(\tilde{x}^{\Delta}(t))^2
 \Delta t \, , \\
\tilde{x}(a)=0 \, , \quad \tilde{x}(b)=0 \, ,
\end{gather*}
which has the solution $\tilde{x}\equiv 0$. For that we consider the
transformation
\begin{equation*}
x(t)=\tilde{x}(t)+ct+d, \quad c,d\in \R,
\end{equation*}
where constants $c$ and $d$ will be chosen later. According to the
above, we have
\begin{equation*}
x^{\Delta}(t)=\tilde{x}^{\Delta}(t)+c, \quad
x^{\sigma}(t)=\tilde{x}^{\sigma}(t)+c\sigma(t)+d
\end{equation*}
and
\begin{equation*}
\begin{split}
L(t,x^{\sigma}(t),x^{\Delta}(t))
&=(x^{\Delta}(t))^2 +x^{\sigma}(t) + t x^{\Delta}(t)\\
&=(\tilde{x}^{\Delta}(t))^2+2c\tilde{x}^{\Delta}(t)+c^2+\tilde{x}^{\sigma}(t)+c\sigma(t)+d
+t\tilde{x}^{\Delta}(t)+ct\\
&=\tilde{L}(t,\tilde{x}^{\sigma}(t),\tilde{x}^{\Delta}(t))
+[2c\tilde{x}(t)+t\tilde{x}(t)+ct^2+(c^2+d)t]^{\Delta}.
\end{split}
\end{equation*}
In order to obtain the solution to the original problem, it suffices
to chose $c$ and $d$ so that
\begin{equation}\label{eq:const:mod}
    \begin{cases}
ca+d=\alpha\\
cb+d=\beta \, .
\end{cases}
\end{equation}
Solving the system of equations \eqref{eq:const:mod} we obtain
$c=\frac{\alpha-\beta}{a-b}$ and $d =\frac{\beta a-b\alpha}{a-b}$.
Hence, the global minimizer to problem \eqref{illust:Ex:mod} is
\begin{equation*}
x(t)=\frac{\alpha-\beta}{a-b}t+\frac{\beta a-b\alpha}{a-b} \, .
\end{equation*}

% --------------------------------------------------

\section{Optimal control on time scales}

The study of more general problems of optimal control on time scales
is in its infancy, and results are rare (see
\cite{Basia:post_doc_Aveiro:2,Z:W:X} for some preliminary results).
Similar to the calculus of variations on time scales, there is a
lack of examples with known solution. Here we solve an optimal
control problem on an arbitrary time scale using the idea of
Leitmann's direct method. Consider the global minimum problem
\begin{equation}
\label{ex1:a} \text{minimize} \quad \mathcal{L}[u_1,u_2] = \int_0^1
\left((u_1(t))^2 + u_2(t))^2\right) \Delta t \\
\end{equation}
subject to the control system
\begin{equation}
\label{ex1:b}
\begin{cases}
x_1^\Delta(t) = \exp(u_1(t)) + u_1(t) + u_2(t) \, ,\\
x_2^\Delta(t) = u_2(t) \, ,
\end{cases}
\end{equation}
and conditions
\begin{equation}
\label{ex1:c}
\begin{gathered}
x_1(0) = 0 \, , \quad x_1(1) = 2 \, , \quad
x_2(0) = 0 \, , \quad x_2(1) = 1 \, , \\
u_1(t) \, , u_2(t) \in \Omega = [-1,1] \, .
\end{gathered}
\end{equation}
This example is inspired in \cite{withLeitmann}. It is worth to
mention that a theory based on necessary optimality conditions on
time scales to solve problem \eqref{ex1:a}-\eqref{ex1:c} does not
exist at the moment.

We begin noticing that problem \eqref{ex1:a}-\eqref{ex1:c} is
variationally invariant according to \cite{GouveiaTorres05} under
the one-parameter transformations\footnote{A computer algebra
package that can be used to find the invariance transformations 
is available from the \emph{Maple Application Center} at
\url{http://www.maplesoft.com/applications/view.aspx?SID=4805}}
\begin{equation}
\label{transfEx1} x_1^s = x_1 + s t \, , \quad x_2^s = x_2 + s t \,
, \quad u_2^s = u_2 + s \quad (t^s = t \text{ and } u_1^s = u_1) \,
.
\end{equation}
To prove this, we need to show that both the functional integral
$\mathcal{L}[\cdot]$ and the control system stay
invariant under the
$s$-parameter transformations \eqref{transfEx1}. This is easily seen by direct calculations:
\begin{equation}
\label{inv:Func:Ex1}
\begin{split}
\mathcal{L}^s[u_1^s,u_2^s]&=
\int_0^1 \left(u_1^{s}(t)\right)^2 + \left(u_2^{s}(t)\right)^2  \Delta t \\
&= \int_0^1 u_1(t)^2 + \left(u_2(t) + s \right)^2 \Delta t \\
&= \int_0^1 \left( u_1(t)^2 + u_2(t)^2  +  [s^2 t + 2 s x_2(t)]^\Delta \right) \Delta t \\
&= \mathcal{L}[u_1,u_2] + s^2 + 2s \, .
\end{split}
\end{equation}
We remark that $\mathcal{L}^s$ and $\mathcal{L}$ have the same
minimizers: adding a constant $s^2 + 2s$ to the functional
$\mathcal{L}$ does not change the minimizer of $\mathcal{L}$. It
remains to prove that the control system also remains invariant
under transformations \eqref{transfEx1}:
\begin{equation}
\label{inv:CS:Ex1}
\begin{split}
\left(x_1^{s}(t)\right)^\Delta&=
\left(x_1(t) + s t \right)^\Delta = x_1^\Delta(t) + s = \exp(u_1(t)) + u_1(t) + u_2(t) + s \\
&= \exp(u_1^s(t)) + u_1^s(t) + u_2^s(t) \, , \\
\left(x_2^{s}(t)\right)^\Delta &= \left(x_2(t) + st \right)^\Delta =
x_2^\Delta(t) + s = u_2(t) + s = u_2^{s}(t) \, .
\end{split}
\end{equation}
Conditions \eqref{inv:Func:Ex1} and \eqref{inv:CS:Ex1} prove that
problem \eqref{ex1:a}-\eqref{ex1:c} is invariant under the
$s$-parameter transformations \eqref{transfEx1} up to $\left(s^2 t +
2s x_2\right)^\Delta$. Using the invariance transformations
\eqref{transfEx1}, we generalize problem \eqref{ex1:a}-\eqref{ex1:c}
to a $s$-parameter family of problems, $s \in \mathbb{R}$, which
include the original problem for $s=0$:
\begin{equation*}
\text{minimize} \quad \mathcal{L}^s[u_1^s,u_2^s] = \int_0^1 (u_1^s(t))^2 +
(u_2^s(t))^2 \Delta t
\end{equation*}
subject to the control system
\begin{equation*}
\begin{cases}
\left(x_1^s(t)\right)^\Delta = \exp(u_1^s(t)) + u_1^s(t) + u_2^s(t) \, ,\\
\left(x_2^s(t)\right)^\Delta = u_2^s(t) \, ,
\end{cases}
\end{equation*}
and conditions
\begin{equation*}
\begin{gathered}
x_1^s(0) = 0 \, , \quad x_1^s(1) = 2 + s \, , \quad
x_2^s(0) = 0 \, , \quad x_2^s(1) = 1 + s \, , \\
u_1^s(t) \in [-1,1] \, ,  \quad u_2^s(t) \in [-1+s,1+s] \, .
\end{gathered}
\end{equation*}
It is clear that $\mathcal{L}^s \geq 0$ and that $\mathcal{L}^s=0$ if
$u_1^{s}(t)=u_2^{s}(t) \equiv 0$. The control equations, the
boundary conditions, and the constraints on the values of the
controls, imply that $u_1^{s}(t)=u_2^{s}(t) \equiv 0$ is admissible
only if $s=-1$: $x_1^{s=-1}(t) = t$, $x_2^{s=-1}(t) \equiv 0$.
Hence, for $s= -1$ the global minimum to $\mathcal{L}^s$ is 0 and the
minimizing trajectory is given by
\begin{equation*}
\tilde{u}_1^{s}(t) \equiv 0  \, , \quad \tilde{u}_2^{s}(t) \equiv 0
\, , \quad \tilde{x}_1^{s}(t)= t  \, , \quad \tilde{x}_2^{s}(t)
\equiv 0 \, .
\end{equation*}
Since for any $s$ one has by \eqref{inv:Func:Ex1} that
$\mathcal{L}[u_1,u_2] = \mathcal{L}^s[u_1^s,u_2^s] - s^2 - 2s$, we
conclude that the global minimum for problem $\mathcal{L}[u_1,u_2]$
is 1. Thus, using the inverse functions of the variational
symmetries \eqref{transfEx1},
\begin{equation*}
u_1(t) = u_1^{s}(t) \, , \quad u_2(t) = u_2^{s}(t)- s \, , \quad
x_1(t) = x_1^{s}(t) - s t \, , \quad x_2(t) = x_2^{s}(t) - s t \, .
\end{equation*}
The absolute minimizer for problem \eqref{ex1:a}-\eqref{ex1:c} is
\begin{equation*}
\tilde{u}_1(t) = 0 \, , \quad \tilde{u}_2(t) = 1 \, , \quad
\tilde{x}_1(t) = 2 t \, , \quad \tilde{x}_2(t) = t \, .
\end{equation*}

%--------------------------------------------------------------------

\section*{Acknowledgements}

Malinowska is a post-doc researcher at the University of Aveiro 
with the support of Bia{\l}ystok University of Technology
via a project of the Polish Ministry of Science and Higher Education. 
Torres was partially supported by the R\&D unit CEOC, via FCT and the EC fund
FEDER/POCI 2010, and partially supported by the research project
UTAustin/MAT/0057/2008.

%----------------------------------------------------------------

% --------------------------------------------

\end{document}